\documentclass[12pt]{amsart}
\usepackage{amsaddr}

\usepackage{latexsym,amsmath,amssymb,amsthm,amsfonts}

\usepackage[numeric]{amsrefs}

\newcommand{\ds}{\displaystyle}
\newcommand{\R}{{\mathbb{R}}}
\newcommand{\eps}{\varepsilon}

\renewcommand{\P}{{\mathcal{P}}}

\newcommand{\x}{{\boldsymbol{x}}}
\renewcommand{\c}{{\boldsymbol{o}}}

\renewcommand{\u}{{\boldsymbol{u}}}
\renewcommand{\v}{{\boldsymbol{v}}}

\newcommand{\y}{{\boldsymbol{y}}}
\newcommand{\z}{{\boldsymbol{z}}}
\renewcommand{\t}{{\boldsymbol{t}}}

\newcommand{\wt}{\widetilde}

\newcommand{\diam}{{\rm diam}}
\newcommand{\meas}{{\rm meas}}

\newcommand{\intr}{{\rm int}}

\renewcommand{\phi}{{\boldsymbol{\varphi}}}
\newcommand{\bpsi}{{\boldsymbol{\psi}}}

\newcommand{\qtq}[1]{{\quad\text{#1}\quad}}

\newtheorem{theorem}{Theorem}[section]
\newtheorem{lemma}[theorem]{Lemma}

\newtheorem{corollary}[theorem]{Corollary}

\theoremstyle{remark}
\newtheorem{remark}[theorem]{Remark}

\numberwithin{equation}{section}

\title{Christoffel function on planar domains with piecewise smooth boundary}

\author{A.\ Prymak}
\address{Department of Mathematics, University of Manitoba, Winnipeg, MB, R3T2N2, Canada\\ e-mail:prymak@gmail.com}
\thanks{The first author was supported by NSERC of Canada Discovery Grant RGPIN 04863-15.}

\author{O.\ Usoltseva}
\address{Department of Mathematics, University of Manitoba, Winnipeg, MB, R3T2N2, Canada\\
e-mail:usoltseo@myumanitoba.ca}
\thanks{The second author was supported by the University of Manitoba Graduate Fellowship and by the Department of Mathematics of the University of Manitoba.}



\begin{document}

\begin{abstract}
	We compute up to a constant factor the Christoffel function on planar domains with boundary consisting of finitely many $C^2$ curves such that each corner point of the boundary has interior angle strictly between $0$ and $\pi$. The resulting formula uses the distances from the point of interest to the curves or certain parts of the curves defining the boundary of the domain.
\end{abstract}

\maketitle

\section{Introduction and main result}

Christoffel function associated with a compact set $D\subset\R^2$ having non-empty interior and with a positive weight function $w\in L_1(D)$ can be defined as
\begin{equation}\label{eqn:classic def}
\lambda_n(D,w,\x)=\left(\sum_{k=1}^Np_k(\x)^2\right)^{-1}, \quad \x\in D,
\end{equation}
where $\{p_k\}_{k=1}^N$ is any orthonormal basis of $\P_{n}$ with respect to the inner product $\langle f,g \rangle = \int_D f(\y)g(\y)w(\y)d\y$, and $\P_{n}$ denotes the space of all real algebraic polynomials of total degree $\le n$ in two variables (so $N=\tfrac{(n+1)(n+2)}{2}$). Christoffel function possesses the following well-known extremal property:
\begin{equation}\label{eqn:def_lambda}
\lambda_n(D,w,\x)=\min_{f\in\P_n,\, f(\x)=1}\int_D f^2(\y)w(\y) d\y, \quad \x\in D.
\end{equation}
For simplicity, for the uniform weight $w\equiv 1$ we set $\lambda_n(D,\x)=\lambda_n(D,w,\x)$.

Christoffel function is a valuable tool in various areas of analysis, mathematics and other disciplines, see, e.g.~\cite{Lu}, \cite{Ne}, \cite{Pa} and~\cite{Si}. A typical result on \emph{asymptotics} of Christoffel function establishes that for any fixed point $\x$ in the interior of $D$ one has $\lim_{n\to\infty} n^2 \lambda_n(D,w,\x)=\Psi(\x)$ where $\Psi(\x)$ is computed explicitly or estimated, and of particular interest is the decay of $\Psi(\x)$ when $\x$ is close to the boundary of $D$. Our focus is on establishing \emph{behavior} of Christoffel function, namely, for any $n$ and any $\x\in D$ we compute $\lambda_n(D,\x)$ up to a constant factor independent of $n$ and $\x$. This implies estimates of $\Psi(\x)$ and is useful in certain applications where the results on asymptotics are not helpful. For example, the rate of growth of $\inf_{\x\in D}(\lambda_n(D,\x))^{-1}$ as $n\to\infty$ was shown in~\cite{Di-Pr} to be determining for Nikol'skii type inequalities on $D$ and in~\cite{Co-Da-Le} to be crucial for stability and accuracy of discrete least squares approximation (see also~\cite{Co-Gi} for the weighted analog employing pointwise behavior of $\lambda_n(D,\x)$).

A common approach to the computation of Christoffel function is to use~\eqref{eqn:classic def} if an orthonormal basis of $\P_{n}$ is available, see, e.g.~\cite{Xu}. This becomes infeasible when $D$ is a rather general multivariate domain and the structure of the orthogonal polynomials on $D$ is unknown. 
A different approach originated in~\cite{Kr} and is based on using~\eqref{eqn:def_lambda} and comparison with other domains for which the behavior of Christoffel function is known, see also~\cite{Di-Pr}, \cite{Pr} and~\cite{Pr-U1}. In this article we further develop this approach and compute the behavior of Christoffel function for a large class of planar domains with piecewise $C^2$ boundary.

Let us define the domains we deal with more precisely. 

A $C^2$ curve is a parametric curve given by a parametrization $\phi$ which is a $C^2$ mapping from $[0,1]$ to $\R^2$ satisfying $|\phi'(s)|\ne0$ everywhere and $\phi(s)\ne\phi(t)$ when $0\le s<t<1$. In particular, we allow closed curves when $\phi(0)=\phi(1)$. By $\partial D$ we denote the boundary of a domain $D\subset \R^2$. We call $D\subset \R^2$ a domain with piecewise $C^2$ boundary if $\partial D=\cup_{i=1}^m\Gamma_i$, where each $\Gamma_i$ is a $C^2$ curve. A point $\v$ is a corner point of $\partial D$ if $v=\phi(s)=\bpsi(t)$ where $\phi$ and $\bpsi$ are parametrizations of $\Gamma_i$ and $\Gamma_\ell$ with $\phi'(s)\ne\bpsi'(t)$. Note that this can happen even when $i=\ell$, $\phi=\bpsi$ say for $t=0$ and $s=1$. However, we do require that if $\Gamma_i$ with parametrization $\phi:[0,1]\to\R^2$ is closed and $\phi(0)$ is not a corner point (i.e. $\phi'(0)=\phi'(1)$), then necessarily $\phi''(0)=\phi''(1)$, so that $\Gamma_i$ is $C^2$ everywhere. 

Let $\{\v_j\}_{j=1}^k$ be the set of all corner points of $\partial D$. For every $j$, we need to define the interior angle $\alpha_j$ of $D$ at $\v_j$, and two related ``corner'' boundary curves $\Gamma_{j}^{\pm}$. Let $\phi:[-\eps,\eps]\to \R^2$, $\eps>0$, be the natural parametrization of $\partial D$ in a neighborhood of $\v_j$ such that $\phi(0)=\v_j$, $D$ remains on the left when $\partial D$ is traversed as the parameter increases and $\Gamma_j^-:=\phi|_{[-\eps,0]}$ and $\Gamma_j^+:=\phi|_{[0,\eps]}$ are $C^2$ curves with no common points except for $\v_j$ (it might be that this corner point is formed at the endpoints of a single $\Gamma_i$). These conditions can be achieved by considering proper orientation and taking sufficiently small $\eps>0$. Then the unit vectors $\t_j^+:=\lim_{s\to0^+}\phi'(s)$ and $\t_j^-:=\lim_{s\to0^-}\phi'(s)$ are tangent to $\partial D$ at $\v_j$. We define $\alpha_j\in[0,2\pi)$ as the angle required to turn $\t_j^+$ counterclockwise to get $-\t_j^-$. If $D$ is a simple polygon, this definition coincides with the standard definition of the interior angle. We remark that in the above we allow piecewise $C^2$ domains to have several connected components and to have holes.

We set $d(\x,\Gamma):=\inf_{\y\in\Gamma}|\x-\y|$ to be the distance from a point $\x$ to a set (or curve) $\Gamma$, where $|\cdot|$ is the Euclidean norm in $\R^2$, and $d(\x,\emptyset):=\infty$. Define $\rho_n^*(t):=n^{-2}+n^{-1}\sqrt{t}$, $t\ge 0$. The equivalence ``$\approx$'' is understood with absolute constants, namely, $A\approx B$ means $c^{-1}A\le B\le cA$ for an absolute constant $c>0$.

Our main result is the following theorem.
\begin{theorem}\label{thm:main}
Let $D\subset\R^2$ be a domain with piecewise $C^2$ boundary with pieces of the boundary $\Gamma_i$, $i=1,\dots,m$, and corner points $\v_j$ with interior angles $\alpha_j$, $0<\alpha_j<\pi$, and related corner boundary curves $\Gamma_j^\pm$ as defined above, $j=1,\dots,k$. For any point $\x\in D$
\begin{equation}\label{eqn:main}
	\lambda_n(D,\x)\approx c(D) \min\Big(\min_{1\le i\le m} n^{-1}\rho_n^*(d(\x,\Gamma_i)) , \min_{1\le j\le k} \rho_n^*(d(\x,\Gamma_j^-)) \rho_n^*(d(\x,\Gamma_j^+)) \Big),
\end{equation} 
where $c(D)>0$ is a constant depending only on $D$.
\end{theorem}

When $\x$ is not near a corner point, the arguments of Theorem~\ref{thm:main} work without the requirement $0<\alpha_j<\pi$, see Remark~\ref{rem:cases12}.

Our proof of Theorem~\ref{thm:main} is based on comparison with appropriate reference domains and use of the extremal property~\eqref{eqn:def_lambda}. For the lower bound, we use ``grain''-type domains which are the intersections of two discs of the same radius. The estimate for such reference domain is established in  Section~\ref{sec:lower} by reduction to use of Videnskii-type inequality~\eqref{eqn:Videnskii}. For the upper bound we compare with the domains which are the intersections of two annuli and explicitly construct  in Section~\ref{sec:upper} the required polynomials with small $L_2$ norm and $f(\x)=1$ at a fixed point $\x$. This construction may be of independent interest as the resulting polynomial can be viewed as a multivariate fast decreasing or ``needle'' polynomial. Known results of this type such as in~\cite{Kr15} and~\cite{Kr16} have radial structure, i.e. the decay estimate is given in terms of the distance to $\x$, which is not suitable for our purposes as $d(\x,\Gamma_j^-)$ and $d(\x,\Gamma_j^+)$ can have different magnitude.  


Our goal was obtaining a description of the Christoffel function which is uniform in $n$ and $\x\in D$. The major portion of the proof is focused on the local situation when $\x$ is close to one of the corner points $\v_j$. We remark that it is not hard to reduce the problem to the case when $\partial D$ has only one boundary piece and one corner point. Namely, for arbitrary $D$ with piecewise $C^2$ boundary and corner point $\v$ of $D$, it is possible to construct two domains $D_1$, $D_2$ with piecewise $C^2$ boundary and the only corner point $\v$ such that $D_1\subset D\subset D_2$ and $D$ coincides with $D_1$ and $D_2$ in a neighborhood of $\v$. We chose the general case of finitely many boundary pieces and global exposition as it is not significantly more complicated and includes the transition between the local and the global behavior. 

The methods of this work allow to handle non-convex domains but do not apply to serve angles bigger than $\pi$ or cusps in the boundary of the domain, which are very interesting directions for future research. 





\section{Preliminaries}

In this section we collect the required preliminaries and introduce some necessary notations. 

In what follows, the constants $c$, $c(\cdot)$, $c_1(\cdot)$, $c_2(\cdot)$, $\dots$ are positive and depend only on parameters indicated in the parentheses (if any). The constants $c$ and $c(\cdot)$ may be different at different occurrences despite the same notation being used. This is in contrast to $c_1(\cdot)$, $c_2(\cdot)$, $\dots$ which have the same value at different occurrences for the same arguments. 

By~\eqref{eqn:def_lambda},
\begin{equation}
\label{eqn:compare}
\text{if }D_1\subset D_2\subset \R^2,\text{ then }
\lambda_n(D_1,\x)\le \lambda_n(D_2,\x), \quad \x\in D_1,
\end{equation}
and
\begin{equation}
\label{eqn:affine}
\lambda_n(TD,T\x)=\lambda_n(D,\x)|\det T|, \quad \x\in D,
\end{equation}
where $T\x=\x_0+A\x$ is any non-degenerate affine transform of $\R^2$, i.e., $\x_0\in\R^2$ and $A$ is a $2\times 2$ matrix, $\det T:=\det A\ne0$.

Let $B:=\{\x:|\x|\le1 \}$ denote the unit disc in $\R^2$. By~\cite[Proposition~2.4 and~(2.3)]{Pr},
\begin{equation}\label{eqn:ball}
\lambda_n(B,\x)\approx n^{-1} \rho_n^*(1-|\x|), \quad \x\in B.
\end{equation}
We will use without reference the following properties of $\rho_n^*$ which are straightforward to verify: $\rho_n^*(t)\le t+ n^{-2}$, $\rho_n^*(t)\le \rho_n^*(t') \le \sqrt{\tfrac{t'}{t}} \rho_n^*(t)$, and $\rho_n^*(t+n^{-2})\approx \rho_n^*(t)$, valid for any $0\le t\le t'$.


Local linear extension of a $C^2$ curve $\Gamma$ having a parametrization $\phi$ is a curve $\Gamma^*$ with $C^1$ parametrization $\phi^*$ satisfying $\phi^*(t)=\phi(t)$ for $t\in[0,1]$, $\phi'(t)=(\phi^*)'(0)$ for $t<0$, and $(\phi^*)'(t)=\phi'(1)$ for $t>1$. We choose the domain of $\phi^*$ to be $[-\epsilon,1+\epsilon]$ for a sufficiently small $\epsilon>0$ so that $\phi^*$ is injective possibly with the exception of $\phi^*(0)=\phi^*(1)$ in case $\phi$ was closed. In other words, $\Gamma^*$ is  obtained by extending $\Gamma$ beyond the beginning point $\phi(0)$ and the end point $\phi(1)$ by straight line segments of strictly positive lengths belonging to the lines tangent to $\Gamma$ at these two points, respectively. 

We need to introduce a notation. Suppose $\Gamma\subset\R^2$ and $\y\in \Gamma$ are such that $\Gamma$ is a $C^2$ curve in a neighborhood of $\y$. Let $\u$ be a unit normal vector to $\Gamma$ at $\y$, and in the case $\Gamma=\partial D$ for a domain $D$, we choose $\u$ to be pointing outwards from $D$. For any $r>0$, we denote by 
\begin{equation}\label{eqn:Bpm def}
B_+(r,\y,\Gamma):=rB+\y+ r\u
\qtq{and} 
B_-(r,\y,\Gamma):=rB+\y- r\u
\end{equation}
the two discs of radius $r$ tangent to $\Gamma$ at $\y$.

We need that closed $C^2$ curves without corner points possess a {\it rolling disc property}. Namely, if $\Gamma$ is such a curve, then there exists  $r=r(\Gamma)>0$ such that for any $0<r'\le r$ and any point $\y\in\Gamma$ we have $B_\pm(r',\y,\Gamma)\cap \Gamma=\{\y\}$. While in a neighborhood of $\y$ such a statement follows from standard differential geometry (curvature is separated from zero), the stated above global version follows from a generalization of Blaschke's rolling theorem~\cite{Wa}*{Th.~1~(iii)~and~(v)}. In the proof of the main result we will extend the rolling disc property to certain non-closed $C^2$ curves.

We will need the following analog of Bernstein and Markov inequalities due to Videnskii~\cite{Vi} (also can be found in~\cite[4.1~E.19, p.~242]{Bo-Er}):
for a trigonometric polynomial $T_n$ of degree $\le n$ and $\theta\in(-\beta,\beta)$, $\beta\in(0,\pi)$ one has
\begin{equation*}
|T_n'(\theta)|\le 
\min\left\{\frac{n\cos\frac{\theta}{2}}{\sqrt{\sin^2\frac{\beta}{2} - \sin^2\frac{\theta}{2}}}, (1+o(1)) 2n^2 \cot\tfrac{\beta}{2} \right\}\|T_n\|_{L_{\infty}([-\beta,\beta])},
\end{equation*}
where $o(1)=0$ for every $n>\tfrac12\sqrt{3\tan^2(\beta/2)+1}$, so $(1+o(1)) 2 \cot\tfrac{\beta}{2}\le c(\beta)$ for all $n$. (In fact, the Markov-type inequality $|T_n'(\theta)|\le c(\beta) n^2 \|T_n\|_{L_{\infty}([-\beta,\beta])}$ was obtained much earlier by Jackson~\cite[p.~889]{Ja}; the Videnskii's inequality has sharp constant for large $n$.) The above implies
\begin{equation}\label{eqn:Videnskii}
|T_n'(\theta)|\le \tilde c(\beta) \frac{\|T_n\|_{L_{\infty}([-\beta,\beta])}}{\rho_n^*(\beta-|\theta|)},
\end{equation}
where $\tilde c(\beta)>0$ can be assumed to be a decreasing function of $\beta\in(0,\pi)$.

We will use the notations $\diam(\cdot)$ and $\intr(\cdot)$ to denote the diameter and the interior of a planar set, respectively; $\meas_d(\cdot)$ will stand for the $d$-dimensional Lebesgue measure. We will also need the distance notation $d(X,Y):=\inf_{\x\in X,\y\in Y}|\x-\y|$ for two subsets $X,Y\subset \R^2$.

\section{Lower bound for specific domains}\label{sec:lower}

The main result of this section is the following lemma establishing an appropriate lower bound for ``grain''-type domain, which is the intersection of two discs of the same radius.
\begin{lemma}\label{lem:grain}
	Let $0<h<2$, $D_1:=B$, $D_2:=B+(0,h)$, $D:=D_1\cap D_2$, for $\x\in D$ $d_i^*(\x):=1-|\x-(i-1)(0,h)|$ is the distance from $\x$ to $\partial D_i$, $i=1,2$. Then for any $\x\in D$
	\[
	\lambda_n(D,\x)\ge c(h) \rho_n^*(d_1^*(\x)) \rho_n^*(d_2^*(\x)).
	\]
\end{lemma}

We need some preparation before the proof of Lemma~\ref{lem:grain}.

A convex body in $\R^2$ is any convex compact set with non-empty interior.
For a $D\subset\R^2$ convex body in $\R^2$, $\x\in D$ and $\mu> 0$, denote by $D_{\mu,\x}:=\x+\mu(D-\x)$ the homothety of $D$ with the ratio $\mu$ and the center $\x$. 
\begin{lemma}\label{lem:convex inclusion}
Let $D\subset\R^2$ be a convex body in $\R^2$, $\x$ be an interior point of $D$ and $0<\mu<1$. Then for any $\y\in D_{1-\mu,\x}$ we have $D_{\mu,\x}-\x+\y\subset D$.
\end{lemma}
\begin{proof}
	It is enough to show that
	\begin{equation*}
	\mu(D-\x)+(1-\mu)(D-\x)+\x \subset D,
	\end{equation*}
	which is immediate since $D$ is convex and so $\mu D+ (1-\mu)D=D$.
\end{proof}

The following corollary allows to control the $L_\infty$ norm of the ``needle'' polynomial realizing $\lambda_n(D,\x)$ (see~\eqref{eqn:def_lambda}) which is not guaranteed to be attained at $\x$.
\begin{corollary}
	\label{col:norm control convex}
	Let $D\subset\R^2$ be a convex body in $\R^2$, $\x$ be an interior point of $D$, $0<\mu<1$, and $P\in\P_{n}$ be a polynomial satisfying $P(\x)=1$ and $\|P\|^2_{L_2(D)}=\lambda_n(D,\x)$. Then
	\begin{equation*}
		\|P\|_{L_\infty(D_{1-\mu,\x})} \le \mu^{-1}.
	\end{equation*}
\end{corollary}
\begin{proof}
Let $M:=\|P\|_{L_\infty(D_{1-\mu,\x})}\ge 1$ be attained at a point $\y\in D_{1-\mu,\x}$. Then
\[
\lambda_n(D,\y) = \min_{Q\in \P_{n}, |Q(\y)|=1} \|Q\|_{L_2(D)}^2 \le \frac1{M^2}\|P\|^2_{L_2(D)} = \frac1{M^2} \lambda_n(D,\x),
\]
so by~\eqref{eqn:compare}, Lemma~\ref{lem:convex inclusion} and~\eqref{eqn:affine}, we conclude that
\begin{align*}
M^2  \le \frac{\lambda_n(D,\x)}{\lambda_n(D,\y)} \le
\frac{\lambda_n(D,\x)}{\lambda_n(D_{\mu,\x}-\x+\y,\y)} = \mu^{-2}.
\end{align*}
\end{proof}

\begin{remark}
	Lemma~\ref{lem:convex inclusion} and Corollary~\ref{col:norm control convex} are valid in $\R^d$ with $\mu^{-1}$ replaced by $\mu^{-d/2}$ in the conclusion of Corollary~\ref{col:norm control convex}. The proofs are exactly the same with $\mu^{-2}$ replaced by $\mu^{-d}$ in the end.
\end{remark}

The restriction of an algebraic polynomial to a circular arc is a trigonometric polynomial. With this in mind, the next lemma employs Videnskii inequality in our settings.
\begin{lemma}
	\label{lem:videnskii arcs}
	Let $A\subset \R^2$ be an arc of length $l$ of a circle of radius $r$, $l<2\pi r$, $\x\in A$, $\eta<l$, $r<\eta^{-1}$, $\eta>0$, $d$ be the distance from $\x$ to the 2-point set consisting of the endpoints of $A$, $f\in\P_n$, $f(\x)=1$. Then
	\begin{equation}
	\label{eqn:value control by Videnskii}
	f(\y)\ge \frac{1}{2} \quad\text{whenever}\quad
	\y\in A\quad\text{and}\quad |\x-\y|\le c(\eta) \|f\|_{L_\infty(A)}^{-1} \rho_n^*(d).
	\end{equation}
\end{lemma}
\begin{proof}
	We can parametrize $A$ so that
	\[ 
	A=\phi([-\beta,\beta]), \quad\text{where}\quad \phi(t)=\c+r(\cos(t-t_0),\sin(t-t_0)),
	\]
	$\c\in\R^2$ and $\beta=\tfrac{l}{2r}$. Note that $T_n(t):=f(\phi(t))$ is a trigonometric polynomial of degree $\le n$. We can assume that $\x$ is closer to $\phi(\beta)$ than to $\phi(-\beta)$, then $\x=\phi(\beta-d')$, where $d'=2\arcsin(\tfrac{d}{2r}) \approx \tfrac{d}{r}$ and $\beta-d'\ge0$. With $\y=\phi(t)$, we will show that
	\begin{equation}\label{eqn:d'}
	f(\y)=T_n(t)\ge\tfrac{1}{2} \quad\text{whenever}\quad
	t\in[-\beta,\beta] \quad\text{and}\quad
	|\beta-d'-t|\le \gamma \|f\|_{L_\infty(A)}^{-1} \rho_n^*(d')
	\end{equation} 
	for a small enough $\gamma=\gamma(\eta)<\tfrac12$. Assuming $t$ is as in~\eqref{eqn:d'}, so by $\gamma<\tfrac12$ we have $|\beta-d'-t|\le\tfrac{d'}2+\tfrac12n^{-2}$ implying $\beta-\tfrac32d'-\tfrac12n^{-2}\le t\le \beta-\tfrac12d'+\tfrac12n^{-2}$. For some $ \theta $ between $\beta-d'$ and $t$, using~\eqref{eqn:Videnskii} we have 
	\begin{align*}
	|1-T_n(t)| = |T_n(\beta-d') - T_n(t)|&=|T_n'(\theta)||\beta-d'-t| \\
	&\le \tilde c(\beta) \frac{\|T_n\|_{L_\infty([-\beta,\beta])}}{\rho_n^*(\beta-|\theta|)} |\beta-d'-t| \\
	&\le c \tilde c(\beta) \frac{\|T_n\|_{L_\infty([-\beta,\beta])}}{\rho_n^*(d')} |\beta-d'-t| \\
	&\le c \tilde c(\beta) \frac{\|f\|_{L_\infty(A)}}{\rho_n^*(d')} \gamma \|f\|_{L_\infty(A)}^{-1} \rho_n^*(d') \le \tfrac12, 
	\end{align*}	
	provided $\gamma=\gamma(\eta)$ is sufficiently small (we have $\tilde c(\beta)\ge \tilde c(\tfrac{\eta^2}{2})$). Finally, \eqref{eqn:d'} implies~\eqref{eqn:value control by Videnskii} because $|\beta-d'-t|r \approx |\x-\y|=|\phi(\beta-d')-\phi(t)|$ and $\rho_n^*(d)\le c(\eta)\rho_n^*(d')$.
\end{proof}

Now we are ready for the proof of the main result of this section.
\begin{proof}[Proof of Lemma~\ref{lem:grain}.]
Without loss of generality, assume $d_1^*(\x)\le d_2^*(\x)$. We will have three cases depending on the values of $d_1^*(\x)$ and $d_2^*(\x)$ in relation to a parameter $\delta=\delta(h)>0$ which will be selected later.

Case~1: $d_1^*(\x)\ge \tfrac\delta2$. Then $\wt B:=\tfrac\delta2 B+\x$ satisfies $\wt B\subset D$, so by~\eqref{eqn:compare}, \eqref{eqn:affine} and~\eqref{eqn:ball}
\begin{align*}
	\lambda_n(D,\x) \ge \lambda_n(\wt B,\x)=\tfrac{\delta^2}{4} \lambda_n(B,(0,0))\approx c(\delta) n^{-2}\ge c(\delta) \rho_n^*(d_1^*(\x)) \rho_n^*(d_2^*(\x)).
\end{align*}

Case~2: $d_1^*(\x)< \tfrac\delta2$ and $d_2^*(\x)\ge\delta$. Let $\y\in\partial B$ be the point closest to $\x$, i.e. $\x=(1-d_1^*(\x))\y$. Now we consider the disc $\wt B:=\tfrac\delta2 B+(1-\tfrac\delta2)\y$, clearly $\wt B\subset B$. Note that due to $d_1^*(\x)< \tfrac\delta2$, the point $\x$ belongs to the line segment joining the center $(1-\tfrac\delta2)\y$ of $\wt B$ and $\y$, in particular, $\x\in \wt B$. Therefore, for any $\z\in\wt B$
\[
|\z-(0,h)|\le |\z-\x|+|\x-(0,h)|\le \delta + 1-d_2^*(\x)\le 1,
\]
implying $\wt B\subset B+(0,h)$ and, consequently, $\wt B\subset D$.
By~\eqref{eqn:compare}, \eqref{eqn:affine} and~\eqref{eqn:ball}
\begin{align*}
\lambda_n(D,\x) \ge \lambda_n(\wt B,\x)=\tfrac{\delta^2}{4} \lambda_n(B,(1-\tfrac{2d_1^*(\x)}{\delta})\y)\approx c(\delta) n^{-1} \rho_n^*(d_1^*(\x)) \ge c(\delta) \rho_n^*(d_1^*(\x)) \rho_n^*(d_2^*(\x)).
\end{align*}

Case~3: $d_2^*(\x)<\delta$. Let us make the required choice of $\delta=\delta(h)>0$ at this time. First, we will impose that
\begin{equation}\label{eqn:arc restriction}
h^2+(1-\delta)^2> 1.
\end{equation}
Next, the set $\partial D_1 \cap \partial D_2$ consists of two points, one of which is $\u_+:=\left(\sqrt{1-(\tfrac{h}{2})^2},\tfrac{h}{2}\right)$. Observe that the set
\[
X(\delta):=\{\y=(y_1,y_2)\in D:y_1\ge0,\ d(\y,\partial D_1)\le\delta,\ d(\y,\partial D_2)\le\delta \}
\]
satisfies $\u_+\in X(\delta)$ and $\ds\lim_{\delta\to0+}\diam(X(\delta))=0$. Therefore, we can choose $\delta=\delta(h)>0$ so that
\begin{equation}\label{eqn:delta choice 0h}
d((0,h),X(\delta))>c_1(h)
\end{equation}
and
\begin{equation}\label{eqn:delta choice separate}
y_1\ge c_1(h) \qtq{and} y_2\ge c_1(h) \qtq{for any} \y=(y_1,y_2)\in X(\delta). 
\end{equation} 

We use Corollary~\ref{col:norm control convex} with $\mu=\tfrac12$. For the dilated polynomial $\wt P(\cdot):=P(2(\cdot-\x)+\x)$ we have
\begin{equation}\label{eqn:new pol}
\lambda_n(D,\x)=\tfrac14\|\wt P\|^2_{L_2(D_{2,\x})}\ge \tfrac14 \|\wt P\|^2_{L_2(D)},
\quad \wt P(\x)=1
\quad\text{and}\quad
\|\wt P\|_{L_\infty(D)}\le 2,
\end{equation}
so to complete the proof it is sufficient to find a set $D'\subset D$ with
\begin{equation}\label{eqn:D'}
\meas_2 (D') \ge c(h) \rho_n^*(d_1^*(\x)) \rho_n^*(d_2^*(\x)) \qtq{and}
\wt P(\y)\ge\tfrac14 \qtq{for every} \y\in D'.
\end{equation}

For any point $\y\in D$, it will be convenient to denote by 
\begin{equation*}
A_i(\y):= D \cap \partial(|\y-(i-1)(0,h)|B+(i-1)(0,h))
\end{equation*}
the largest arc of the circle concentric with $\partial D_i$, passing through $\y$ and located inside $D$, $i=1,2$. 
Since $\x\in X(\delta)$, by~\eqref{eqn:delta choice 0h}, the length of $A_2(\x)$ is at least $c(h)$ (the radius is clearly $\le 1$), so by Lemma~\ref{lem:videnskii arcs}, there is a choice of $\gamma_1=\gamma_1(h)$ such that $\wt P(\y)\ge\tfrac{1}{2}$ for any $\y\in A_3$, where
\begin{equation*}
A_3:=\{\y\in A_2(\x):|\x-\y|\le \gamma_1 \rho_n^*(d_1^*(\x)) \}.
\end{equation*}
In addition, we can assume that $\gamma_1\cdot(1+\delta)<\tfrac{c_1(h)}{2}$ so that by $\x\in X(\delta)$, $\rho_n^*(\delta)\le n^{-2}+\delta$ and~\eqref{eqn:delta choice separate}, we have
\begin{equation}\label{eqn:gamma 1 choice}
y_1\ge \tfrac12c_1(h) \qtq{and} y_2\ge \tfrac12c_1(h) \qtq{for any} \y=(y_1,y_2)\in A_3, 
\end{equation}  
in particular,
\begin{equation}\label{eqn:norm y in A3}
\inf(\{|\y|:\y\in A_3 \}) \ge \tfrac{1}{\sqrt{2}} c_1(h).
\end{equation}
We will show that
\begin{equation}\label{eqn:A3 r coord bound}
\meas_1(\{|\y|:\y\in A_3 \}) \ge c(h) \rho_n^*(d_1^*(\x)).
\end{equation}
Note that the circle containing $A_3$ has its largest first coordinate at the point $(1-d_2^*(\x),h)$, which, by~\eqref{eqn:arc restriction}, is outside of $D_1$. Therefore, $A_3$ is the graph of an increasing function (of the first coordinate), so we can denote the endpoints of $A_3$ as $(z_1,z_2)$ and $(z_1+v_1,z_2+v_2)$, where $v_1>0$ and $v_2>0$. By the definition of $A_3$, $|(v_1,v_2)|\approx c(h) \rho_n^*(d_1^*(\x))$. Now~\eqref{eqn:A3 r coord bound} follows from~\eqref{eqn:gamma 1 choice} and $A_3\subset B$ as follows:
\begin{align*}
\meas_1(\{|\y|:\y\in A_3 \}) &\ge |(z_1+v_1,z_2+v_2)| - |(z_1,z_2)|
=\frac{2z_1v_1+2z_2v_2+v_1^2+v_2^2}{|(z_1+v_1,z_2+v_2)| + |(z_1,z_2)|} \\
&\ge z_1 v_1+ z_2 v_2 \ge \tfrac{c_1(h)}{2}(v_1+v_2) \ge c(h) \rho_n^*(d_1^*(\x)). 
\end{align*}

By~\eqref{eqn:norm y in A3}, the length of $A_1(\y)$ is at least $c(h)$ for any $\y\in A_3$ while the radius is clearly $\le 1$, so by Corollary~\ref{col:norm control convex}, we can choose $\gamma_2$ so that $\wt P(\z)\ge \tfrac14$ for any $\z\in A_4(\y)$, where
\begin{equation}\label{eqn:A_4 def}
A_4(\y):=\{\z\in A_1(\y): |\y-\z|\le \gamma_2 \rho_n^*(d_2^*(\x)) \}.
\end{equation}
Remark that this choice of $\gamma_2$ is independent of the choice of $\y\in A_3$ as it only depends on $c_1(h)$ in~\eqref{eqn:norm y in A3}. 

By the construction, $\wt P(\z)\ge \tfrac14$ for any $\z\in D'$, where
$
\ds D':=\bigcup_{\y\in A_3} A_4(\y),
$ 
so it remains to justify the first inequality in~\eqref{eqn:D'}. This can be conveniently seen using polar coordinates $(r,\theta)$. For any point in $D'$ we have $r\ge c(h)$ by~\eqref{eqn:norm y in A3}. For each fixed $r$, the measure of $\theta$ in $D'$, which is the length of some arc $A_4(\y)$, $\y\in A_3$, is at least $c(h)\rho_n^*(d_2^*(\x))$ by~\eqref{eqn:norm y in A3} and~\eqref{eqn:A_4 def}. The lower bound for the measure of $r$ in $D'$ is provided by~\eqref{eqn:A3 r coord bound}. This completes the proof. 
\end{proof}


\section{Lemmas for upper bounds}\label{sec:upper}

We begin with a basic univariate construction.
\begin{lemma}\label{lem:DPpol}
For any $t\in[0,1]$ and any positive integer $n$ there exists $Q\in\P_{n/2,1}$ such that $Q(1-t)=1$ and 
\begin{equation}\label{eqn:univariate pol estimate}
|Q(s)|\le c \frac{\rho_n^*(t)}{\rho_n^*(t)+|1-t-s|}, \quad s\in[-1,1].
\end{equation}
\end{lemma}
\begin{proof}
	This is a partial case of~\cite[Lemma~6.1]{Di-Pr}.
\end{proof}

Using rotation, we get a good radial polynomial on an annulus.
\begin{lemma}\label{lem:lower basic polynomial}
For any $0<r_1<r_2$ let $D:=\{\x\in\R^2: r_1\le |\x|\le r_2 \}$. Then for any $\lambda\in[r_1,r_2]$ there exists $P\in\P_{n}$ such that
\begin{equation}\label{eqn:basic pol equals 1}
P(\y)=1  \qtq{whenever} |\y|=\lambda
\end{equation}  
and 
\begin{equation}\label{eqn:basic pol estimate}
|P(\y)|\le c(r_1,r_2)\frac{\rho_n^*(\lambda-r_1)}{\rho_n^*(\lambda-r_1)+|\lambda-|\y||}, \quad \y\in D.
\end{equation}
\end{lemma}
\begin{proof}
	Set $t:=(\lambda^2-r_1^2)/(r_2^2-r_1^2)\in[0,1]$. Let $Q$ be the polynomial provided by Lemma~\ref{lem:DPpol}.
	We define $P(\x)=Q((r_2^2-|\x|^2)/(r_2^2-r_1^2))$, clearly $P$ satisfies~\eqref{eqn:basic pol equals 1}. 
	The inequality~\eqref{eqn:basic pol estimate} now follows from~\eqref{eqn:univariate pol estimate} using that $\rho_n^*(t)\approx c(r_1,r_2)\rho_n^*(\lambda-r_1)$.
\end{proof}

The radial structure from the previous lemma will carry too large overall $L_2$ norm, which can be rectified by multiplication by a good univariate polynomial. 
\begin{lemma}\label{lem:lower basic polynomial narrowed}
	For any $0<r_1<r_2$ let $D:=\{\x\in\R^2: r_1\le |\x|\le r_2 \}$. Then 
	\begin{equation*}
	\lambda_n(D,\x)\le c(r_1,r_2) n^{-1} \rho_n^*(|\x|-r_1) \qtq{for any} \x\in D.
	\end{equation*}
\end{lemma}
\begin{proof}
Due to~\eqref{eqn:affine}, it is enough to show for any fixed $\lambda\in[r_1,r_2]$ that $$\lambda_n(D,(\lambda,0))\le c(r_1,r_2)n^{-1}\rho_n^*(\lambda-r_1),$$ or, due to~\eqref{eqn:def_lambda}, that there exists $\wt P\in\P_{n}$ satisfying 
\[\wt P(\lambda,0)=1 \qtq{and} \|\wt P\|^2_{L_2(D)}\le c(r_1,r_2)n^{-1}\rho_n^*(\lambda-r_1).\]

Let $P\in \P_{n/2}$ be the polynomial provided by Lemma~\ref{lem:lower basic polynomial} satisfying~\eqref{eqn:basic pol equals 1} and~\eqref{eqn:basic pol estimate}. Let $Q$ be the polynomial supplied by Lemma~\ref{lem:DPpol} for $t=1$. We define $\wt P(x_1,x_2):=P(x_1,x_2)Q(\tfrac{x_2}{r_2})$. Then clearly $\wt P(\lambda,0)=1$ and we need to show that
\begin{equation*}
\int_D (P(x_1,x_2))^2(Q(\tfrac{x_2}{r_2}))^2\,dx_1dx_2 \le c(r_1,r_2)n^{-1}\rho_n^*(\lambda-r_1).
\end{equation*}
By~\eqref{eqn:univariate pol estimate} and~\eqref{eqn:basic pol estimate}, using polar coordinates, it is sufficient to prove that
\begin{equation*}
\int_{r_1}^{r_2} \left( \frac{\rho_n^*(\lambda-r_1)}{\rho_n^*(\lambda-r_1)+|\lambda-r|} \right)^2
\int_0^{2\pi}
\left(\frac{\tfrac1n}{\tfrac1n+|\tfrac{r\sin\theta}{r_2}|} \right)^2 d\theta \, r\, dr
\le c(r_1,r_2)n^{-1}\rho_n^*(\lambda-r_1).
\end{equation*}
Observe that $r\approx c(r_1,r_2)$ for any $r\in[r_1,r_2]$. With $I:=\{\theta\in[0,2\pi]: |\sin\theta|\le\tfrac1n \}$ we have
\begin{align*}
\int_0^{2\pi}
\left(\frac{\tfrac1n}{\tfrac1n+|\tfrac{r\sin\theta}{r_2}|} \right)^2 d\theta
& \le c(r_1,r_2) \left[ \int_{I} d\theta + \int_{[0,2\pi]\setminus I}  \left(\frac{\tfrac1n}{|\sin\theta|} \right)^2 d\theta \right] \le c(r_1,r_2) n^{-1},
\end{align*}
for any $r\in[r_1,r_2]$. Therefore, it remains to show that 
\begin{equation*}
\int_{r_1}^{r_2} \left( \frac{\rho_n^*(\lambda-r_1)}{\rho_n^*(\lambda-r_1)+|\lambda-r|} \right)^2
dr
\le c(r_1,r_2)\rho_n^*(\lambda-r_1),
\end{equation*}
which can be done using the same idea as for the integral w.r.t. $\theta$, namely, splitting the integral as $\int_J\dots+\int_{[r_1,r_2]\setminus J}\dots$, where $J=\{r\in[r_1,r_2]:|\lambda-r|\le\rho_n^*(\lambda-r_1) \}$. We omit the details as they are the same as in~\eqref{eqn:fi bound} from the proof of Lemma~\ref{lem:lower prod of two basic}.
\end{proof}

The final lemma provides a good polynomial on the intersection of two annuli obtained by multiplication of two polynomials provided by Lemma~\ref{lem:lower basic polynomial}.
\begin{lemma}\label{lem:lower prod of two basic}
	Let $\zeta>0$, $\zeta\le h\le 1-\zeta$, $\c_i:=(-1)^i(0,r_1h)$, $0<r_1<r_2$, $D_i:=\{\x\in\R^2: r_1\le |\x-\c_i|\le r_2 \}$, $i=1,2$, $D:=D_1\cap D_2$. Suppose $\x\in D$ is such that $|\x-\c_i|\le (1+\tfrac\zeta2)r_1$, $i=1,2$. Then
	\begin{equation*}
	\lambda_n(D,\x) \le c(\zeta,r_1,r_2) \rho_n^*(|\x-\c_1|-r_1) \rho_n^*(|\x-\c_2|-r_1).
	\end{equation*}
\end{lemma}
\begin{proof}
	Due to~\eqref{eqn:affine}, for simplicity, we will assume that $r_1=1$ and $r=\tfrac{r_2}{r_1}=r_2>1$. Fix $\x\in D$, denote $t_i:=|\x-\c_i|-1$. We apply Lemma~\ref{lem:lower basic polynomial} to find a polynomial $P_i\in\P_{n/2}$ such that 
	\begin{equation}\label{eqn:Pi value 1}
	P_i(\y)=1 \qtq{whenever} |\y-\c_i|=t_i+1
	\end{equation}
	and
	\begin{equation}\label{eqn:Pi props}
	|P_i(\y)|\le c(r) \frac{\rho_n^*(t_i)}{\rho_n^*(t_i)+|t_i+1-|\y-\c_i||}.
	\end{equation}
	We will prove that $P(\y):=P_1(\y)P_2(\y)$ is a required polynomial (see~\eqref{eqn:def_lambda}), namely, it satisfies $P(\x)=1$ and $\|P\|_{L_2(D)}^2\le c(\zeta,r) \rho_n^*(t_1)\rho_n^*(t_2)$. The equality $P(\x)=1$ is immediate by~\eqref{eqn:Pi value 1}. Due to~\eqref{eqn:Pi props}, we need to show that 
	\begin{align*}
	\int_D \frac{\rho_n^*(t_1)}{(\rho_n^*(t_1)+|t_1+1-|\y-\c_1||)^2} \cdot \frac{\rho_n^*(t_2)}{(\rho_n^*(t_2)+|t_2+1-|\y-\c_2||)^2} \, d\y &=:\int_D f_1(\y)\cdot f_2(\y)\,d\y \\ &\le c(\zeta,r).
	\end{align*}
	Observing that the integrand is symmetric about the second coordinate axis (containing both $\c_i$), it is sufficient to estimate the integral over $D_+:=\{\y=(y_1,y_2)\in D:y_1\ge0\}$, which will be split into three parts. First, let $D^*:=\{\y=(y_1,y_2)\in D: y_1\ge \zeta'\}$, where $\zeta'=\min\{\tfrac{\sqrt{1-\zeta^2}}{2},\zeta\}$. Consider the mapping $\Phi:D^*\to[1,r]^2$ defined by $\Phi(y_1,y_2):=(|\y-\c_1|,|\y-\c_2|)$. It is easily seen that $\Phi$ is injective on $D^*$. We also have that $\Phi$ is continuously differentiable on $D^*$ and its Jacobian determinant $J_\Phi$ satisfies 
	\begin{equation*}
	J_\Phi(y_1,y_2)=\frac{2y_1h}{|\y-\c_1||\y-\c_2|}\ge \frac{2 \zeta' \zeta}{r^2}  =: c_2(\zeta,r)
	\end{equation*}
	on $D^*$. Therefore, with $\u=(u_1,u_2)$ and $\ds\tilde f_i(u_i)=\frac{\rho_n^*(t_i)}{(\rho_n^*(t_i)+|t_i+1-u_i|)^2}$, we have
	\begin{align*}
	\int_{D^*} f_1(\y)\cdot f_2(\y)\,d\y & = 
	\int_{\Phi(D^*)} \tilde f_1(u_1) \cdot \tilde f_2(u_2) \frac{1}{J_\Phi(\Phi^{-1}(u_1,u_2))} \, d\u \\
	& \le \frac1{c_2(\zeta,r)}  \int_1^r \tilde f_1(u_1) \,du_1 \cdot \int_1^r \tilde f_2(u_2) \,du_2.
	\end{align*}
	Now let $I_i:=\{u_i\in[1,r]: |t_i+1-u_i|\le \rho_n^*(t_i) \}$. We obtain
	\begin{align}
	\int_1^r \tilde f_i(u_i) \,du_i &=
	\int_{I_i}  \tilde f_i(u_i) \,du_i + \int_{[1,r]\setminus I_i}  \tilde f_i(u_i) \,du_i \nonumber \\
	&\le \int_{I_i}  \frac1{\rho_n^*(t_i)} \,du_i + 2\int_{\rho_n^*(t_i)}^\infty \frac{\rho_n^*(t_i)}{(2s)^2}\, ds \label{eqn:fi bound} \\
	&\le 2 + \frac12, \nonumber
	\end{align}
	so $\int_{D^*} f_1(\y)\cdot f_2(\y)\,d\y \le c(\zeta,r)$. It remains to prove that $\int_{D_+\setminus D^*} f_1(\y)\cdot f_2(\y)\,d\y \le c(\zeta,r)$. Note that for $J:=[\sqrt{1-\zeta'^2}+h , r+h]$ we have
	$
	D_+\setminus D^* \subset [0, \zeta']\times ((-J)\cup J).
	$
	With fixed $y_1\in [0,\zeta']$, first we will show
	\begin{equation}\label{eqn:int J}
	\int_{J} f_1(y_1,y_2)\cdot f_2(y_1,y_2)\,dy_2 \le c(\zeta,r).
	\end{equation} 
	For $y_2\in J$ we have 
\begin{align*}
	|\y-\c_1| &\ge y_2+h\ge \sqrt{1-\zeta'^2}+2h\ge \sqrt{1-\zeta^2}+2\zeta \\
	&>\sqrt{(1-\zeta)^2}+2\zeta=1+\zeta\ge |\x-\c_1|+\tfrac\zeta2 = t_1+1+\tfrac\zeta2,
\end{align*}
so
	\begin{equation}\label{eqn:f1 bound}
	f_1(y_1,y_2)\le \frac{\rho_n^*(t_1)}{(|\y-\c_1|-t_1-1)^2}\le \frac{\rho_n^*(t_1)}{(\tfrac\zeta2)^2}\le c(\zeta).
	\end{equation}
	We further observe that for $y_2\in J$ one has 
\begin{equation*}
	\sqrt{1-\zeta'^2}\le y_2-h\le |\y-\c_2|\le y_1+(y_2-h)\le \zeta'+r
\end{equation*}	
implying 
	\begin{equation*}
	\frac{\partial}{\partial y_2}|\y-\c_2|=\frac{y_2-h}{|\y-\c_2|}\ge \frac{\sqrt{1-\zeta'^2}}{r+\zeta'}=:c_3(\zeta,r)
	\qtq{and}
	\frac{\partial}{\partial y_2}|\y-\c_2|\le \frac{r}{\sqrt{1-\zeta'^2}}=:c_4(\zeta,r), 
	\end{equation*}
	so
	\begin{equation}\label{eqn:lip}
	||(y_1,y_2)-\c_2|-|(y_1,y_2')-\c_2||\approx c(\zeta,r) |y_2-y_2'|
	\qtq{for any}
	y_2,y_2'\in J.
	\end{equation}
	Now denote $J_2:=\{ y_2\in J : |t_2+1-|\y-\c_2||\le \rho_n^*(t_2) \}$, by~\eqref{eqn:lip} we have $\meas_1(J_2)\approx c(\zeta,r)\rho_n^*(t_2)$,
	and so
	\begin{equation*}
	\int_{J_2} f_2(y_1,y_2)\,dy_2 \le \int_{J_2} \frac1{\rho_n^*(t_2)} \,dy_2 \le c(\zeta,r).
	\end{equation*}
	Let $y_2'\in J$ be such that $|(y_1,y_2')-\c_2|=t_2+1$. Using~\eqref{eqn:lip}, we have with sufficiently small $c_5(\zeta,r)$ that 
	\begin{align*}
	\int_{J\setminus J_2} f_2(y_1,y_2)\,dy_2 &\le
	\frac14\int_{J\setminus J_2} \frac{\rho_n^*(t_2)}{(t_2+1-|\y-\c_2|)^2} \,dy_2
	\le c(\zeta,r)  \int_{J\setminus J_2} \frac{\rho_n^*(t_2)}{(y_2'-y_2)^2} \,dy_2 \\
	&\le c(\zeta,r) \int_{\{y_2\in J:|y_2'-y_2|\ge  c_5(\zeta,r) \rho_n^*(t_2) \}} \frac{\rho_n^*(t_2)}{(y_2'-y_2)^2} \,dy_2 \\
	&\le c(\zeta,r) \int_{\rho_n^*(t_2)}^\infty \frac{\rho_n^*(t_2)}{s^2}\, ds = c(\zeta,r).
	\end{align*}
	In summary, $\int_J f_2(y_1,y_2)\,dy_2\le c(\zeta,r)$, and taking~\eqref{eqn:f1 bound} into account, we obtain~\eqref{eqn:int J}. Note that~\eqref{eqn:int J} is valid if $J$ is replaced with $-J$ by following essentially the same proof, namely, one simply interchanges $i=1$ and $i=2$ observing that $|(y_1,y_2)-\c_i|=|(y_1,-y_2)-\c_{3-i}|$. Therefore 
	\begin{equation*}
	\int_{D_+\setminus D^*} f_1(\y)\cdot f_2(\y)\,d\y
	\le
	\int_{[0,\zeta']\times((-J)\cup J)} f_1(\y)\cdot f_2(\y)\,d\y \le c(\zeta,r),
	\end{equation*}
	and the proof is complete.
	\end{proof}

\section{Proof of the main result}

\begin{proof}[Proof of Theorem~\ref{thm:main}]

Denote $d_i(\x):=d(\x,\Gamma_i)$, $i=1,\dots,m$, and $d_j^{\pm}(\x):=d(\x,\Gamma_j^\pm)$, $j=1,\dots,k$. Observe that for every $j$ and choice of $\pm$, there exists $i$ such that $\Gamma_j^\pm\subset \Gamma_i$ and then $d_j^\pm(\x)\ge d_i(\x)$.

We consider three cases depending on a parameter $\delta=\delta(D)>0$ which will be selected later.

Case~1: $d_i(\x)\ge \tfrac\delta4$ for all $i=1,\dots,m$. Then $\tfrac\delta4 B+\x\subset D$, so by~\eqref{eqn:compare}, \eqref{eqn:affine} and~\eqref{eqn:ball}
\begin{align*}
\lambda_n(D,\x) \ge \lambda_n(\tfrac\delta4 B+\x,\x)=\tfrac{\delta^2}{16} \lambda_n(B,(0,0))\approx c(\delta) n^{-2}.
\end{align*}
In the other direction, let $R:=\diam(D)$, then $D\subset RB+\x$ and again  by~\eqref{eqn:compare}, \eqref{eqn:affine} and~\eqref{eqn:ball}
\begin{align*}
\lambda_n(D,\x) \le \lambda_n(R B+\x,\x)=R^2 \lambda_n(B,(0,0))\approx c(R) n^{-2}.
\end{align*}
We complete Case~1 by observing that
\begin{equation*}
\min\Big(\min_{1\le i\le m} n^{-1}\rho_n^*(d_i(\x)) , \min_{1\le j\le k} \rho_n^*(d_j^-(\x)) \rho_n^*(d_j^+(\x)) \Big)\approx c(\delta,R) n^{-2}
\end{equation*}
due to $d_i, d_j^-, d_j^+ \in [\tfrac\delta4,R]$ for any $i,j$.

Case~2: $\min_{1\le i\le m}d_i(\x)=d_{i_0}(\x)<\tfrac\delta4$ for some $i_0$ and $d(\x,\{\v_j\}_{j=1}^k)\ge\delta$. Let $\Gamma_i'$ be the curve obtained by removing open $\tfrac{\delta}{4}$-neighborhoods of every corner point from $\Gamma_i$. Then $\Gamma_i'$ is either a closed $C^2$ curve without corner points (if nothing was removed from $\Gamma_i$), in which case we set $\Gamma_i'':=\Gamma_i'$, or $\Gamma_i'$ can be extended to a closed $C^2$ curve $\Gamma_i''$ without corner points (we can use the connected component of $\partial D$ containing $\Gamma_i'$ and modify $\partial D$ to be $C^2$ smooth in arbitrarily small neighborhood of each corner point). Let $r(\Gamma_i'')$ be a radius fulfilling the rolling disc property for $\Gamma_i''$. With
\begin{equation*}
r_0:=\min\big(\min_{1\le i\le m} r(\Gamma_i''),\tfrac13 \min_{1\le i_1<i_2\le m} d(\Gamma_{i_1}',\Gamma_{i_2}') \big)
\end{equation*}
and $\Gamma':=\cup_{i=1}^m \Gamma_i'$ we have the following extended rolling disc property:
\begin{equation}\label{eqn:case 2 rolling disc}
B_\pm(r,\y,\Gamma')\cap \Gamma'=\{\y\}
\qtq{for any} 0<r\le r_0
\qtq{and any} \y\in\Gamma'.
\end{equation}
When $\delta$ is selected, we will impose
\begin{equation}\label{eqn:r0 vs delta new}
{\frac{\delta}4}\le r_0.
\end{equation}
Let $\y\in\Gamma_{i_0}$ be a point such that $|\x-\y|=d(\x,\Gamma_{i_0})$, by $d(\x,\{\v_j\}_{j=1}^k)\ge\delta$ clearly $\y\in\Gamma'$ and $\x-\y$ is orthogonal to $\Gamma_{i_0}$ at $\y$. We have~\eqref{eqn:case 2 rolling disc} with $r=\tfrac{\delta}{4}$, but we claim that we have even stronger $B_\pm(\tfrac{\delta}{4},\y,\partial D)\cap \partial D=\{\y\}$. Indeed, if $\z\in\partial D\setminus \Gamma'$, then $d(\z,\{\v_j\}_{j=1}^k)<\tfrac{\delta}4$ and $|\y-\z|\ge|\x-\z|-|\x-\y|> \tfrac{3\delta}{4} - \tfrac{\delta}{4}=\diam(B_\pm(\tfrac{\delta}{4},\y,\partial D))$, as required.

Clearly $B_\pm(\tfrac{\delta}{4},\y,\partial D)\cap \partial D=\{\y\}$ implies $B_-(\tfrac\delta4,\y,\partial D)\subset D$. Let $\u$ be the outward unit normal vector to $\partial D$ at $\y$. By~\eqref{eqn:compare}, \eqref{eqn:affine} and~\eqref{eqn:ball}
\begin{align*}
\lambda_n(D,\x) & \ge \lambda_n(B_-(\tfrac\delta4,\y,\partial D),\x)=\tfrac{\delta^2}{16} \lambda_n(B,(1-\tfrac{d_{i_0}(\x)}{\delta/4})\u) \\ 
& \approx c(\delta) n^{-1} \rho_n^*(d_{i_0}(\x)).
\end{align*}
For the other direction, we have that $B_+(\tfrac\delta4,\y,\partial D) \cap \intr(D)=\emptyset$. Further, (recall that $R=\diam(D)$) $D\subset 2R B +\y+\tfrac\delta4\u$, so for
\begin{equation*}
\wt D:=\{\z\in\R^2: \tfrac\delta4 \le |\y+\tfrac\delta4\u-\z|\le 2R \}
\end{equation*}
we have $D\subset \wt D$ and by~\eqref{eqn:compare} and Lemma~\ref{lem:lower basic polynomial narrowed}
\begin{equation*}
\lambda_n(D,\x) \le \lambda_n(\wt D,\x) \le c(\delta, R) n^{-1} \rho_n^*(d_{i_0}(\x)).
\end{equation*}
To arrive at~\eqref{eqn:main} we need to justify that for any $j$ 
\begin{equation}\label{eqn:case2 final bound}
\rho_n^*(d_j^-(\x)) \rho_n^*(d_j^+(\x)) \ge c(D) n^{-1} \rho_n^*(d_{i_0}(\x)).
\end{equation}
Indeed, since $\Gamma_j^+\cap\Gamma_j^-=\{\v_j\}$ there exists $\delta'>0$ depending on $D$ (and $\delta$) such that $\max\{d_j^+(\z),d_j^-(\z)\}<\delta'$ implies $d(\z,\{\v_j\}_{j=1}^k)<\tfrac\delta4$. Therefore $\max\{d_j^+(\x),d_j^-(\x)\}\ge \delta'$, but we also have $\min\{d_j^+(\x),d_j^-(\x)\}\ge d(\x,\partial D)=d_{i_0}(\x)$, so~\eqref{eqn:case2 final bound} follows.

Case~3: $d(\x,\{\v_j\}_{j=1}^k)<\delta$. This is the main case which requires some preparation.

For every $j$ and each choice of $\pm$ we let $\Gamma_j^{\pm*}$ be some local linear extension of $\Gamma_j^\pm$ satisfying $(\Gamma_j^{\pm*}\setminus \Gamma_j^\pm)\cap D=\emptyset$, which is possible to achieve since all interior angles are less than $\pi$. Such a choice of local linear extensions depends only on $D$. Every $\Gamma_j^{\pm*}$ can be extended to a closed $C^2$ curve, so by the rolling disc property,  
\begin{equation}\label{eqn:rolling disc for extensions}
B_\pm(r,\y,\Gamma_j^{\pm*})\cap \Gamma_j^{\pm*}=\{\y\}
\qtq{for any} 0<r\le \tilde r
\qtq{and any} \y\in\Gamma_j^{\pm*},
\end{equation}
where $\tilde r$ is the smallest radius fulfilling the rolling disc property for all (finitely many) curves $\Gamma_j^{\pm*}$. Similarly to Case~2, we will impose
\begin{equation}\label{eqn:r0 vs delta}
\gamma\delta\le \tilde r
\end{equation}
with some $\gamma=\gamma(D)\ge 2$ which will be selected later.

Next we describe and establish certain properties (i)--(iv) of $\partial D$ and $\Gamma_j^{\pm*}$ required from the choice of $\delta$. 



{\sl Property~(i).} Set $\delta_1=\delta_1(D)$ to be the smallest length of the line segments  $\Gamma_j^{\pm*}\setminus \Gamma_j^{\pm}$ and of the curves $\Gamma_j^{\pm}$ (over all $j$ and choices of $\pm$). Let $\y_{\pm}(\x)$ be a point from $\Gamma_j^{\pm*}$ such that $d(\x,\Gamma_j^{\pm*})=|\x-\y_{\pm}(\x)|$. If $\delta\le\delta_1$, the above implies that for $\x$ in $\delta$-neighborhood of $\v_j$ the point $\y_{\pm}(\x)$ is not an endpoint of $\Gamma_{j}^{\pm*}$ so that $\x-\y_{\pm}(\x)$ is orthogonal to the unit tangent vector of $\Gamma_{j}^{\pm}$ at $\y_{\pm}(\x)$. Moreover, $\ds|\y_{\pm}(\x)-\v_j|\le \delta$.

{\sl Property~(ii).} There exists $\delta_2=\delta_2(D)$ such that for any $\delta\le\delta_2$ the $(\gamma+1)\delta$-neighborhood $U$ of $\v_j$ satisfies $U \cap \Gamma_{j'}^{\pm*} =\emptyset$ for $j'\ne j$ (possible by $(\Gamma_{j'}^{\pm*}\setminus \Gamma_{j'}^\pm)\cap D=\emptyset$) and $U \cap \partial D \subset \Gamma_j^+\cup\Gamma_j^-$.

For properties~(iii) and~(iv) we assume that $\delta\le\min\{\delta_1,\delta_2\}$ and $\x$ is in $\delta$-neighborhood of $\v_j$ for an arbitrary $j$.

{\sl Property~(iii).} Let $\u_\pm(\y)$ be the unit normal vector to $\Gamma_j^{\pm*}$ at $\y\in\Gamma_j^{\pm*}$ chosen in continuous manner so that $\u_\pm(\y)$ points outward of $D$ when $\y\in\partial D$. Since $D$ is a $C^2$ domain, the angle between $\u_\pm(\y)$ and $\u_\pm(\y')$ does not exceed $c(D)|\y-\y'|$, for any $\y,\y'\in\Gamma_j^{\pm*}$. 
Combining this with property~(i), we can ensure that the angle between $\u_\pm(\y_\pm(\x))$ and $\u_\pm(\v_j)$ is less than $\tfrac\eps3$ whenever $\delta\le\delta_3=\delta_3(D)$, where $\ds\eps=\min_j \{\alpha_j,\pi-\alpha_j\}$ (recall that $\alpha_j$ is the interior angle of $D$ at $\v_j$).  Consequently, the angle between $\u_{+}(\y_{+}(\x))$ and $\u_{-}(\y_{-}(\x))$ is at least $\tfrac{\eps}{3}$ and at most $\pi-\tfrac\eps3$.

{\sl Property~(iv).} There exists $\delta_4=\delta_4(D)$ such that for any $\delta\le\delta_4$ 
\begin{equation*}
\intr(B_+(\gamma\delta,\y_{\pm}(\x),\Gamma_{j}^{\pm*})) \cap \Gamma_j^\mp =\emptyset
\end{equation*}
which is rather obvious by $0<\alpha_i<\pi$. (We make the same choice of either top or bottom sign in each $\pm$ or $\mp$.) 

We take $\delta:=\min\{\delta_1,\dots,\delta_4,4r_0,\tfrac{\tilde r}{\gamma}\}$, where $\gamma=\gamma(\eps)$ will be selected later and then~\eqref{eqn:r0 vs delta new}, \eqref{eqn:r0 vs delta} and properties~(i)--(iv) are satisfied.

Now we are ready for the proof. We let $d_\pm^*(\x):=d(\x,\Gamma_j^{\pm*})$ where $j$ is such that $d(\x,\{\v_j\}_{j=1}^k)=|\x-\v_j|$. 
We have $\y_\pm(\x)\in\Gamma_j^{\pm*}$ is not an endpoint of $\Gamma_j^{\pm*}$ by property~(i). Let $\wt D_\pm:=B_-(2\delta,\y_{\pm}(\x),\Gamma_j^{\pm*})$. By~\eqref{eqn:rolling disc for extensions}, \eqref{eqn:r0 vs delta}, $\gamma\ge2$ and property~(ii), $\wt D_{\pm}\cap \Gamma_j^{\pm*}=\emptyset$ and so $\wt D:=\wt D_+\cap \wt D_-$ (which contains $\x$) satisfies $\wt D\subset D$. Observe that $\wt D$ is the intersection of two discs of the same radius $2\delta$. We intend to apply Lemma~\ref{lem:grain} to an affine image of $\wt D$. Let $\c_\pm$ be the center of $\wt D_\pm$. We note that $\x$ belongs to each line segment joining $\c_\pm$ and $\y_\pm(\x)$, moreover, $l_\pm:=|\c_\pm-\x|=2\delta-d_\pm^*(\x)\in[\delta,2\delta]$. Due to property~(iii), these two line segments intersect (at $\x$) at an angle $\theta$ which is between $\tfrac\eps3$ and $\pi-\tfrac\eps3$. Let $\theta_\pm$ be the angle opposite to $l_\pm$ in the triangle with the vertices at $\x$, $\c_+$ and $\c_-$. Without loss of generality, we can assume $d_+^*(\x)\ge d_-^*(\x)$.  Then $\theta_-\ge\theta_+$ and $\theta_-+\theta_+=\pi-\theta$. We have
\begin{equation*}
\tilde h:=|\c_--\c_+| = \frac{\sin \theta}{\sin \theta_-}l_- \ge \frac{\sin\tfrac\eps3}{1} \delta.
\end{equation*}
If $\theta_-\ge\tfrac\pi2$, then $\tilde h\le l_-\le 2\delta$. Otherwise,
\begin{equation*}
\tilde h=\frac{\sin \theta}{\sin \theta_-}l_-\le \frac{\sin\theta}{\sin\tfrac{\pi-\theta}2} 2\delta = 2(\sin\tfrac{\theta}{2}) 2\delta \le (\cos\tfrac\eps6) 4\delta.
\end{equation*}
There is an affine transform $T$ which is a composition of appropriate rotation, translation and homothety with ratio $(2\delta)^{-1}$ such that Lemma~\ref{lem:grain} is applicable to $T(\wt D)$. We obtain the required lower bound on $\lambda_n(D,\x)$ in the standard manner using~\eqref{eqn:compare} and~\eqref{eqn:affine}:
\begin{align*}
\lambda_n(D,\x) & \ge\lambda_n(\wt D,\x)= 4\delta^2 \lambda_n(T(\wt D),T(\x)) \\ &\ge c(D) \rho_n^*(d_-^*(\x))\rho_n^*(d_+^*(\x)) \ge c(D) \rho_n^*(d_-(\x))\rho_n^*(d_+(\x)),
\end{align*}
where in the last step we used that $\sin(\tfrac{2\eps}{3}) d_\pm(\x)\le d_\pm^*(\x)$. Indeed, if $ d_\pm(\x)\ne  d_\pm^*(\x)$ for some choice of $\pm$, then by property~(i) $d_\pm(\x)=|\x-\v_j|$ and the claimed inequality follows from properties~(i) and~(iii).

For the upper bound, we will apply Lemma~\ref{lem:lower prod of two basic}. With slight abuse/reintroduction of notations, it will be convenient to use the same notations as for the lower bound now related to different discs and sets. Let $\wt D_\pm:=B_+(\gamma\delta,\y_\pm(\x),\Gamma_j^{\pm*})$. By~\eqref{eqn:rolling disc for extensions} and~\eqref{eqn:r0 vs delta}, $\intr(\wt D_\pm)\cap \Gamma_j^{\pm*}=\emptyset$. Moreover, by property~(iv) $\intr(\wt D_\pm)\cap \Gamma_{j}^{\mp}=\emptyset$. So taking property~(ii) into account, we obtain $\intr(\wt D_-\cup \wt D_+)\cap D=\emptyset$. Let $\c_\pm$ be the center of $\wt D_\pm$. We have (due to $\delta\le \diam(D)=R$) $D\subset 2RB+\c_\pm$, so we will apply Lemma~\ref{lem:lower prod of two basic} to an appropriate affine image of the set 
\begin{equation*}
\wt D:=((2RB+\c_-)\setminus \wt D_-) \cap ((2RB+\c_+)\setminus \wt D_+)
\end{equation*}
containing $D$ and get the proper estimate. Let us first justify the conditions of Lemma~\ref{lem:lower prod of two basic}. Arguing similarly to the lower bound, we let $l_\pm:=|\c_\pm-\x|=\gamma\delta+d_\pm^*(\x)\in[\gamma\delta,(\gamma+1)\delta]$. Due to property~(iii), the lines containing $l_i$ intersect (at $\x$) at an angle $\theta$ which is between $\tfrac\eps3$ and $\pi-\tfrac\eps3$. Let $\theta_\pm$ be the angle opposite to $l_\pm$ in the triangle with the vertices at $\x$, $\c_-$ and $\c_+$, then (recall the assumption $d_+^*(\x)\ge d_-^*(\x)$) $\theta_-\le\theta_+$ and $\theta_-+\theta_+=\pi-\theta$. We have
\begin{equation*}
\tilde h:=|\c_--\c_+| = \frac{\sin \theta}{\sin \theta_+}l_+ \ge \frac{\sin\tfrac\eps3}{1} \gamma\delta.
\end{equation*}
If $\theta_+\ge\tfrac\pi2$, then $\tilde h\le l_+\le (\gamma+1)\delta=\tfrac{\gamma+1}{2\gamma} \cdot 2\gamma\delta$. Otherwise,
\begin{equation*}
\tilde h=\frac{\sin \theta}{\sin \theta_+}l_+\le \frac{\sin\theta}{\sin\tfrac{\pi-\theta}2} (\gamma+1)\delta = 2(\sin\tfrac{\theta}{2}) (\gamma+1)\delta \le \frac{\gamma+1}{\gamma}(\cos\tfrac\eps6) 2\gamma\delta.
\end{equation*}
It remains to choose a large enough $\gamma=\gamma(\eps)$ so that for some $\zeta\in(0,\tfrac12)$ we have $\tfrac{\gamma+1}{\gamma}(\cos\tfrac\eps6)<1-\zeta$, $\sin\tfrac\eps3\ge\zeta$, $\tfrac{\gamma+1}{2\gamma}<1-\zeta$ and $\tfrac{\gamma+1}{\gamma}\le1+\tfrac\zeta2$ (the last condition will ensure that $|\x-\c_\pm|\le(1+\tfrac\zeta2) \gamma\delta$).
Now with appropriate $T$ which is a composition of a rotation and a translation and with $r_1=\gamma\delta$ and $r_2=2R$ by Lemma~\ref{lem:lower prod of two basic}, \eqref{eqn:compare} and~\eqref{eqn:affine}
\begin{align*}
\lambda_n(D,\x) & \le \lambda_n(\wt D,\x)=\lambda_n(T(\wt D),T(\x))\le c(D) \rho_n^*(d_-^*(\x))\rho_n^*(d_+^*(\x))\\ &\le c(D) \rho_n^*(d_-(\x))\rho_n^*(d_+(\x)).
\end{align*}
To establish~\eqref{eqn:main}, observe that by property~(ii) for $j'\ne j$ we have $d(\x,\Gamma_{j'}^\pm)\ge \delta$.
\end{proof}

\begin{remark}\label{rem:cases12}
	The arguments of Case~1 and Case~2 do not use the hypothesis that $0<\alpha_j<\pi$, so under the conditions of Theorem~\ref{thm:main} without this hypothesis, for any $\delta>0$ and any $\x\in D$ such that $d(\x,\{\v_j\}_{j=1}^k)\ge\delta$, we have
	\begin{equation*}
	\lambda_n(D,\x)\approx c(\delta,D) \min_{1\le i\le m} n^{-1}\rho_n^*(d(\x,\Gamma_i)).
	\end{equation*} 
\end{remark}

{\bf Acknowledgment.} We are grateful to the referee for the careful reading of the manuscript and several valuable suggestions that pointed out some inaccuracies and, more importantly, led to an improvement of the generality of the result.

\end{document}